\documentclass[reqno]{amsart}

\usepackage{graphicx}
\usepackage{amscd}
\usepackage{amsmath, amssymb}

\newtheorem{theorem}{Theorem}

\newtheorem{definition}[theorem]{Definition}

\begin{document}

\title{ON THE IRREDUCIBILITY OF THE EXTENSIONS OF BURAU AND GASSNER REPRESENTATIONS}
\author{Mohamad N. Nasser \and Mohammad N. Abdulrahim}

\address{Mohamad N. Nasser\\
         Department of Mathematics and Computer Science\\
         Beirut Arab University\\
         P.O. Box 11-5020, Beirut, Lebanon}
\email{m.nasser@bau.edu.lb}

\address{Mohammad N. Abdulrahim\\
         Department of Mathematics and Computer Science\\
         Beirut Arab University\\
         P.O. Box 11-5020, Beirut, Lebanon}
\email{mna@bau.edu.lb}

\maketitle

\begin{abstract}

Let $Cb_n$ be the group of basis conjugating automorphisms of a free group $\mathbb{F}_n$, and $C_n$ the group of conjugating automorphisms of $\mathbb{F}_n$. Valerij G. Bardakov has constructed representations of $Cb_n$, $C_n$ in the groups $GL_n(\mathbb{Z}[{t_1}^{\pm1}, \ldots ,{t_n}^{\pm 1}])$ and in $GL_n(\mathbb{Z}[{t}^{\pm1}])$ respectively, where $t_1, \ldots, t_n, t$ are indeterminate variables. We show that these representations are reducible and we determine the irreducible components of the representations in $GL_n(\mathbb{C})$, which are obtained by giving values to the variables above. Next, we consider the tensor product of the representations of $Cb_n$, $C_n$ and study their irreduciblity in the case $n=3$.

\end{abstract}

\medskip

\renewcommand{\thefootnote}{}

\vskip 0.1in

\section{Introduction} 

The braid group on $n$ strings, $B_n$, is the abstract group with generators $\sigma_1,\ldots,\sigma_{n-1}$ and a presentation as follows:
\begin{align*}
&\sigma_i\sigma_{i+1}\sigma_i = \sigma_{i+1}\sigma_i\sigma_{i+1} ,\hspace{0.5cm} i=1,2,\ldots,n-2,\\
&\sigma_i\sigma_j = \sigma_j\sigma_i , \hspace{2.25cm} |i-j|>2.
\end{align*}

The pure braid group, $P_n$, is defined as the kernel of the homomorphism $B_n \mapsto S_n$ defined by $\sigma_i \mapsto (i \hspace{0.2cm} i+1)$, $1\leq i \leq n-1$, where $S_n$ is the symmetric group of $n$ elements.

The most famous linear representation of $B_n$ is Burau representation [4], and the most famous linear representation of $P_n$ is Gassner representation [3].

One of the generalizations of the braid group $B_n$ is the group $C_n$ of conjugating automorphisms of $\mathbb{F}_n$, the free group of rank $n$ with the generators $x_1,\ldots ,x_n$ (see [6]). Here $C_n$ is defined to be the subgroup of $Aut(\mathbb{F}_n)$ that satisfies for any $\phi \in C_n$, $\phi(x_i)={f_i}^{-1}x_{\Pi(i)}f_i$, where $\Pi$ is a permutation on $\{1, 2,\ldots, n\}$ and $f_i$ lies in $\mathbb{F}_n$. By the Theorem of Artin [3], the group $B_n$ admits a faithful representation in $Aut(\mathbb{F}_n)$ such that an automorphism $\beta$ satisfies the following two conditions:
\begin{itemize}
\item[(1)] $\beta(x_i)={f_i}^{-1}x_{\Pi(i)}f_i$, \hspace{1cm} $1 \leq i \leq n$,
\item[(2)]  $\beta(x_1x_2\ldots x_n)=x_1x_2\ldots x_n$,
\end{itemize}
where $\Pi$ is a permutation on $\{1,2,\ldots,n\}$ and $f_i\in \mathbb{F}_n$. Recall that condition (1) is the defining condition for an automorphism of $\mathbb{F}_n$ to be in $C_n$, the group of conjugating automorphisms.

Also, one of the generalizations of the pure braid group $P_n$ is the group of basis conjugating automorphisms $Cb_{n}$ [5], which is the subgroup of $C_n$ that satisfies for any $\phi \in Cb_{n}$, $\phi(x_i)=f_i^{-1}x_if_i$, where $f_i\in \mathbb{F}_n$.

$P_n$ is a normal subgroup of $B_n$ and $Cb_n$ is a normal subgroup of $C_n$. In addition, the quotient groups $B_n\slash P_n$ and $C_n\slash Cb_n$ are isomorphic to $S_n$. A. G. Savuschkina [6] proved that $C_n$ is a semidirect product $C_n=Cb_n \leftthreetimes S_n$.

Denote $\mathbb{F}_n^{\prime}=[\mathbb{F}_n,\mathbb{F}_n],$ the commutator subgroup of $\mathbb{F}_n$, and $\mathbb{A}_n=\mathbb{F}_n / \mathbb{F}_n^{\prime}$. The natural map from $Aut(\mathbb{F}_n)$ into $Aut(\mathbb{A}_n)$ is an epimorphism. The kernel of this map is the group of IA-automorphisms denoted by $IA(\mathbb{F}_n)$ (see [1]).

We consider $Cb_n$ as a subgroup of $IA(\mathbb{F}_n)$, the group of IA-automorphisms of the group $\mathbb{F}_n$.

In [2], Bardakov uses Magnus representation defined in [3, Ch. 3] to construct a linear representation $\rho: IA(\mathbb{F}_n) \mapsto GL_n(\mathbb{Z}[{t_1}^{\pm1},\ldots,{t_n}^{\pm 1}])$. Restricting the representation $\rho$ to $Cb_n$ we obtain a representation $\hat{\rho}_G$, which is an extension of Gassner representation of $P_n$. Putting $t_1=\ldots=t_n$ in the representation $\hat{\rho}_G$, we obtain a representation $\hat{\rho}_B$ of $C_n$, which is an extension of Burau representation of $B_n$.

We study, in section 3, the irreducibility of the representation $\hat{\rho}_G$. We prove that $\hat{\rho}_G$  is reducible (Theorem 3). In order not to get a one-dimensional representation, we assume that one of the $t_i$'s not one. Without loss of generality, we set $t_n\neq 1$. We prove that the complex specialization of its $(n-1)$th degree composition factor $\hat{\phi}_G$ is irreducible if and only if $t_i\neq 1$ for all $1\leq i < n$ (Theorem 4).

Similarly, we study in section 4 the irreducibility of the representation $\hat{\rho}_B$. We prove that $\hat{\rho}_B$ is reducible (Theorem 6). Also we prove that the complex specialization of its $(n-1)$th degree composition factor $\hat{\phi}_B$ is irreducible (Theorem 7).

In section 5, we prove, for $n=3$, that the tensor product representation $\hat{\phi}_G(t_1,t_2,t_3) \otimes \hat{\phi}_G(m_1,m_2,m_3)$ is irreducible if and only if $(t_1,t_2,t_3)$ and $(m_1,m_2,m_3)$ are distinct vectors (Theorem 8).

In section 6, we prove, for $n=3$, that the tensor product representation $\hat{\phi}_B(t) \otimes \hat{\phi}_B(m)$ is irreducible if and only if $t \neq m$ (Theorem 9).

\section{Preliminaries} 

The group of conjugating automorphisms, $C_n$, is the subgroup of $Aut(\mathbb{F}_n)$ that satisfies for any $\phi \in C_n$, $\phi(x_i)={f_i}^{-1}x_{\Pi(i)}f_i$, where $\Pi$ is a permutation on $\{1,2,\ldots,n\}$ and $f_i\in \mathbb{F}_n$.

The group of basis conjugating automorphisms, $Cb_{n}$, is the subgroup of $C_n$ that satisfies for any $\phi \in Cb_{n}$, $\phi(x_i)=f_i^{-1}x_if_i$, where $f_i\in \mathbb{F}_n$.

J. McCool [5] proved that the group ${Cb}_n$ is generated by the automorphisms 
$$ \epsilon_{ij}:
\left\{\begin{array}{l}
x_i \mapsto {x_j}^{-1}x_ix_j$,\hspace{0.5cm} $i\neq j\\ 
x_l \mapsto x_l$, \hspace{1.5cm} $l\neq i,\
\end{array}\right.$$
where $1\leq i\neq j \leq n$.

Recall that $IA(\mathbb{F}_n)$ is generated by the automorphisms $\epsilon_{ij}$, $1\leq i\neq j\leq n$ and the automorphisms 
$$ \epsilon_{ijk}:
\left\{\begin{array}{l}
x_i \mapsto x_i[x_j,x_k]$, \hspace{0.5cm} $k\neq i,j\\ 
x_l \mapsto x_l$, \hspace{1.6cm} $l\neq i,\\
\end{array}\right.$$
where $[a,b]=a^{-1}b^{-1}ab$ [2].

In [6], we have $C_n=Cb_n \leftthreetimes S_n$. This means that $C_n$ is generated by the automorphisms $\epsilon_{ij}$, where $1\leq i\neq j \leq n$, and the permutations $\alpha_i$ where $1\leq i \leq n-1$. Here $\alpha_i$ is defined as follows: $$\alpha_i= \left\{\begin{array}{l}
x_i \mapsto x_{i+1} \\
x_{i+1} \mapsto x_i,\hspace{0.5cm} i=1,2,\ldots,n-1\\
x_j \mapsto x_j, \hspace{0.8cm} j\neq i,i+1
\end{array}\right. $$

\begin{definition}
\text{[2]}
The group $IA(\mathbb{F}_n)$ is the group of the IA-automorphisms of the group $\mathbb{F}_n$. We introduce the representation $\rho:IA(\mathbb{F}_n) \mapsto GL_n(\mathbb{Z}[{t_1}^{\pm1},\ldots ,{t_n}^{\pm 1}])$ as follows:
\begin{align*}
&\epsilon_{ij} \mapsto \rho(\epsilon_{ij}):
\left\{\begin{array}{l}
e_i\rho(\epsilon_{ij})={t_j}^{-1}(t_i-1)e_j+{t_j}^{-1}e_i,\\
e_l\rho(\epsilon_{ij})=e_l$, \hspace{0.5cm} $l\neq i,\\
\end{array}\right.\\
\\
&\epsilon_{ijk} \mapsto \rho(\epsilon_{ijk}):
\left\{\begin{array}{l}
e_i\rho(\epsilon_{ijk})=e_i+t_i{t_j}^{-1}({t_k}^{-1}-1)e_j+t_i{t_k}^{-1}(1-{t_j}^{-1})e_k,\\
e_l\rho(\epsilon_{ijk})=e_l$, \hspace{0.5cm} $l\neq i.\\
\end{array}\right.
\end{align*}
\end{definition}
Here we consider the matrices $\rho(\epsilon_{ij})$ and $\rho(\epsilon_{ijk})$ as automorphisms of $W_n$, a free left $R-$module with basis $\{e_1,e_2,\ldots,e_n\}$, where $R= \mathbb{Z}[{t_1}^{\pm1},\ldots ,{t_n}^{\pm 1}]$. Throughout our work, we consider $GL_n(R)$ as acting from the left on column vectors and acting from the right on row vectors.

\section{The irreducibility of the representations $\hat{\rho}_G$} 

\begin{definition}
\text{[2]}
The representation $\hat{\rho}_G$ is defined by
$$\hat{\rho}_G:Cb_{n} \mapsto GL_n(\mathbb{Z}[t_1^{\pm1},\ldots,t_n^{\pm1}])$$
$$\epsilon_{ij} \mapsto
\left( \begin{array}{c|@{}c|c@{}}
   \begin{matrix}
     I_{i-1} 
   \end{matrix} 
      & 0 & 0 \\
      \hline
    0 &\hspace{0.2cm} \begin{matrix}
   		t_j^{-1} & 0 & \dots  & \dots& \dots&  t_j^{-1}(t_i-1) \\
   		0 & 1  & 0 & \dots & \dots & 0 \\
   		0 & 0 & 1 & 0 & \dots & 0\\
   		\vdots & \vdots & & \ddots & & \vdots \\
   		0 & \dots & \dots & 0 & 1 & 0 \\
   		0 &  0 &\dots & \dots & 0 & 1  \\
\end{matrix}  & 0  \\
\hline
0 & 0 & I_{n-j}
\end{array} \right) \hspace*{0.2cm} for \hspace*{0.2cm} i<j,$$ 

$$\epsilon_{ij} \mapsto
\left( \begin{array}{c|@{}c|c@{}}
   \begin{matrix}
     I_{j-1} 
   \end{matrix} 
      & 0 & 0 \\
      \hline
    0 &\hspace{0.2cm} \begin{matrix}
   		1 & 0 & \dots  & \dots& 0 &  0 \\
   		0 & 1  & 0 & \dots & \dots & 0 \\
   		\vdots & \vdots & & \ddots & & \vdots \\
   		0 & \dots & \dots & 1 & 0 & 0\\
   		0 & \dots & \dots & 0 & 1 & 0 \\
   		t_j^{-1}(t_i-1) &  0 &\dots & \dots & 0 & t_j^{-1}  \\
\end{matrix}  & 0  \\
\hline
0 & 0 & I_{n-i}
\end{array} \right) \hspace*{0.2cm} for \hspace*{0.2cm} j<i.$$ 
\end{definition}
Note that in the cases $i=1$ and $j=1$ we omit the first $i-1$ (respectively $j-1$) rows and $i-1$ (respectively $j-1$) columns. And in the cases $i=n$ and $j=n$ we omit the last $n-i$ (respectively $n-j$) rows and $n-i$ (respectively $n-j$) columns.

\begin{theorem}
The representation $\hat{\rho}_G$ is reducible.
\end{theorem}

\begin{proof}
Let $v=[t_1-1,t_2-1,\ldots,t_n-1]^T$, where $T$ is the transpose. We see that $\epsilon_{ij}(v)=v$ for all $1\leq i,j\leq n$, and so $v$ is fixed under the generators of  $\hat{\rho}_G$. Thus $\hat{\rho}_G$ is reducible.
\end{proof}
We specialize $t_1,\ldots, t_n$ to non-zero complex numbers. We want to find a composition factor of degree $n-1$ of $\hat{\rho}_G$. We may assume, in order not to get a one-dimensional representation, that not all $t_i$'s take on the value one. This means that there exists $t_j\neq 1$ for $1\leq j\leq n$. For the complex vector space $\mathbb{C}^n$ of dimension $n$, we consider the basis $S=\{e_1,\ldots, e_{j-1},\underbrace{e_{j+1},\ldots,e_n}_{n-j},v\}$, where $v=[t_1-1,t_2-1,\ldots,t_n-1]^T$.  It is clear that $S$ is a basis of $\mathbb{C}^n$ as $t_j\neq 1$. Now, to make calculations easier, we assume,  without loss of generality, that $j=n$ and so $t_n\neq 1$. In this way, the basis $S$ is $\{e_1,\ldots, e_{n-1}, v\}$.\\
\underline{For $i<j \neq n$}: \vspace{0.2cm} \\
$\epsilon_{ij}(e_1)=e_1$, $\epsilon_{ij}(e_2)=e_2$, \dots, $\epsilon_{ij}(e_{i-1})=e_{i-1}$, $\epsilon_{ij}(e_{i})=t_j^{-1}e_i$, $\epsilon_{ij}(e_{i+1})=e_{i+1}$, \dots, $\epsilon_{ij}(e_{j-1})=e_{j-1}$, $\epsilon_{ij}(e_j)=t_j^{-1}(t_i-1)e_i+e_j$, $\epsilon_{ij}(e_{j+1})=e_{j+1}$, \dots, $\epsilon_{ij}(e_{n-1})=e_{n-1}$, $\epsilon_{ij}(v)=v$.\\
\underline{For $j=n$}: \vspace{0.2cm} \\
$\epsilon_{in}(e_1)=e_1$, $\epsilon_{in}(e_2)=e_2$, \dots, $\epsilon_{in}(e_{i-1})=e_{i-1}$, $\epsilon_{in}(e_{i})=t_n^{-1}e_i$, $\epsilon_{in}(e_{i+1})=e_{i+1}$, \dots, $\epsilon_{in}(e_{n-1})=e_{n-1}$, $\epsilon_{in}(v)=v$.\\
\underline{For $j<i \neq n$}:  \vspace{0.2cm} \\
$\epsilon_{ij}(e_1)=e_1$, $\epsilon_{ij}(e_2)=e_2$, \dots, $\epsilon_{ij}(e_{j-1})=e_{j-1}$, $\epsilon_{ij}(e_{j})=e_{j}+t_j^{-1}(t_i-1)e_i$, $\epsilon_{ij}(e_{j+1})=e_{j+1}$, \dots, $\epsilon_{ij}(e_{i-1})=e_{i-1}$, $\epsilon_{ij}(e_{i})=t_j^{-1}e_i$, $\epsilon_{ij}(e_{i+1})=e_{i+1}$, \dots, $\epsilon_{ij}(e_{n-1})=e_{n-1}$, $\epsilon_{ij}(v)=v$.\\
\underline{For $i=n$}:  \vspace{0.2cm} \\
$\epsilon_{nj}(e_1)=e_1$, $\epsilon_{nj}(e_2)=e_2$, \dots, $\epsilon_{nj}(e_{j-1})=e_{j-1}$, $\epsilon_{nj}(e_j)=-t_j^{-1}(t_1-1)e_1-t_j^{-1}(t_2-1)e_2- ... -t_j^{-1}(t_{j-1}-1)e_{j-1}+t_j^{-1}e_j-t_j^{-1}(t_{j+1}-1)e_{j+1}- ... -t_j^{-1}(t_{n-1}-1)e_{n-1}+t_j^{-1}v$, $\epsilon_{nj}(e_{j+1})=e_{j+1}$, \dots, $\epsilon_{nj}(e_{n-1})=e_{n-1}$, $\epsilon_{nj}(v)=v.$

So, the representation $\hat{\rho}_G$, in the new basis $S$, becomes
$$\epsilon_{ij} \mapsto
\left( \begin{array}{c|@{}c|c@{}}
   \begin{matrix}
     I_{i-1} 
   \end{matrix} 
      & 0 & 0 \\
      \hline
    0 &\hspace{0.2cm} \begin{matrix}
   		t_j^{-1} & 0 & \dots  & \dots& \dots& 0 \\
   		0 & 1  & 0 & \dots & \dots & 0 \\
   		0 & 0 & 1 & 0 & \dots & 0\\
   		\vdots & \vdots & & \ddots & & \vdots \\
   		0 & \dots & \dots & 0 & 1 & 0 \\
   		 t_j^{-1}(t_i-1) &  0 &\dots & \dots & 0 & 1  \\
\end{matrix}  & 0  \\
\hline
0 & 0 & I_{n-j}
\end{array} \right) \hspace*{0.2cm} for \hspace*{0.2cm} i<j \neq n,$$ 

$$\epsilon_{in} \mapsto
\left( \begin{array}{c|@{}c|c@{}}
   \begin{matrix}
     I_{i-1} 
   \end{matrix} 
      & 0 & 0 \\
      \hline
    0 &\hspace{0.2cm} \begin{matrix}
   		t_n^{-1} \\
\end{matrix}  & 0  \\
\hline
0 & 0 & I_{n-i}
\end{array} \right),$$

 $$\epsilon_{ij} \mapsto
\left( \begin{array}{c|@{}c|c@{}}
   \begin{matrix}
     I_{j-1} 
   \end{matrix} 
      & 0 & 0 \\
      \hline
    0 &\hspace{0.2cm} \begin{matrix}
   		1 & 0 & \dots  & \dots& 0 &  t_j^{-1}(t_i-1) \\
   		0 & 1  & 0 & \dots & \dots & 0 \\
   		\vdots & \vdots & & \ddots & & \vdots \\
   		0 & \dots & \dots & 1 & 0 & 0\\
   		0 & \dots & \dots & 0 & 1 & 0 \\
   		0 &  0 &\dots & \dots & 0 & t_j^{-1}  \\
\end{matrix}  & 0  \\
\hline
0 & 0 & I_{n-i}
\end{array} \right) \hspace*{0.2cm} for \hspace*{0.2cm} j<i \neq n,$$ 

$$\epsilon_{nj} \mapsto
\left( \begin{array}{c|@{}c@{}}
   \begin{matrix}
     I_{j-1} 
   \end{matrix} 
      & 0 \\
      \hline
    \begin{matrix}
   		q_1 & \dots  & q_{j-1} \\
   		0 & 0  & 0 \\
   		\vdots  & \vdots  & \vdots \\
   		0 &  0 & 0  \\
   		0 &  0 & 0  \\

\end{matrix}  &\hspace{0.2cm} \begin{matrix}
   		t_j^{-1} & q_{j+1}  & \dots  &  q_{n-1} & t_j^{-1}\\
   		0 & 1  & 0 & \dots & 0 \\
   		\vdots & & \ddots  & &\vdots \\
   		0 & \dots & 0  & 1 & 0 \\
   		0 & \dots & \dots  & 0 & 1 \\
\end{matrix} \\
\end{array} \right),$$ 
where $q_k=-t_j^{-1}(t_k-1)$ for all $1 \leq k \neq j \leq n-1$.
\\

Now, we remove the last row and the last column to obtain the $n-1$ composition factor $\hat{\phi}_G$ given by the following generators
$$\epsilon_{ij} \mapsto
\left( \begin{array}{c|@{}c|c@{}}
   \begin{matrix}
     I_{i-1} 
   \end{matrix} 
      & 0 & 0 \\
      \hline
    0 &\hspace{0.2cm} \begin{matrix}
   		t_j^{-1} & 0 & \dots  & \dots& \dots& 0 \\
   		0 & 1  & 0 & \dots & \dots & 0 \\
   		0 & 0 & 1 & 0 & \dots & 0\\
   		\vdots & \vdots & & \ddots & & \vdots \\
   		0 & \dots & \dots & 0 & 1 & 0 \\
   		 t_j^{-1}(t_i-1) &  0 &\dots & \dots & 0 & 1  \\
\end{matrix}  & 0  \\
\hline
0 & 0 & I_{n-j-1}
\end{array} \right) \hspace*{0.2cm} for \hspace*{0.2cm} i<j \neq n,$$ 

$$\epsilon_{in} \mapsto
\left( \begin{array}{c|@{}c|c@{}}
   \begin{matrix}
     I_{i-1} 
   \end{matrix} 
      & 0 & 0 \\
      \hline
    0 &\hspace{0.2cm} \begin{matrix}
   		t_n^{-1} \\
\end{matrix}  & 0  \\
\hline
0 & 0 & I_{n-i-1}
\end{array} \right),$$

 $$\epsilon_{ij} \mapsto
\left( \begin{array}{c|@{}c|c@{}}
   \begin{matrix}
     I_{j-1} 
   \end{matrix} 
      & 0 & 0 \\
      \hline
    0 &\hspace{0.2cm} \begin{matrix}
   		1 & 0 & \dots  & \dots& 0 &  t_j^{-1}(t_i-1) \\
   		0 & 1  & 0 & \dots & \dots & 0 \\
   		\vdots & \vdots & & \ddots & & \vdots \\
   		0 & \dots & \dots & 1 & 0 & 0\\
   		0 & \dots & \dots & 0 & 1 & 0 \\
   		0 &  0 &\dots & \dots & 0 & t_j^{-1}  \\
\end{matrix}  & 0  \\
\hline
0 & 0 & I_{n-i-1}
\end{array} \right) \hspace*{0.2cm} for \hspace*{0.2cm} j<i \neq n,$$ 

$$\epsilon_{nj} \mapsto
\left( \begin{array}{c|@{}c@{}}
   \begin{matrix}
     I_{j-1} 
   \end{matrix} 
      & 0 \\
      \hline
    \begin{matrix}
   		q_1 & \dots  & q_{j-1} \\
   		0 & 0  & 0 \\
   		\vdots  & \vdots  & \vdots \\
   		0 &  0 & 0  \\
   		0 &  0 & 0  \\

\end{matrix}  &\hspace{0.2cm} \begin{matrix}
   		t_j^{-1} & q_{j+1}  & \dots  & \dots &  q_{n-1}\\
   		0 & 1  & 0 & \dots & 0 \\
   		\vdots & & \ddots  & &\vdots \\
   		0 & \dots & 0  & 1 & 0 \\
   		0 & \dots & \dots  & 0 & 1 \\
\end{matrix} \\
\end{array} \right),$$ 
where $q_k=-t_j^{-1}(t_k-1)$ for all $1 \leq k \neq j \leq n-1$.

\vspace{0.3cm}
Now, we consider the complex specialization of $\hat{\phi}_G$ by letting $t_i$ be non-zero complex numbers for all $1\leq i \leq n$ and $t_n\neq 1$.
\begin{theorem}
Let $0\neq t_1, \ldots, t_n \in \mathbb{C}$ and $t_n\neq 1$. The representation $\hat{\phi}_G(t_1, \ldots, t_n): Cb_{n} \mapsto GL_{n-1}(\mathbb{C})$ is irreducible if and only if $t_i \neq 1$ for all $1\leq i< n$.
\end{theorem}
\begin{proof}
For the necessary condition, suppose that there exists $1\leq s < n$ such that $t_s=1$.\\
Since $s\neq n$, it follows that:
\begin{itemize}
\item for $i\neq n$ and $1 \leq j \leq n$, $\hat{\phi}_G(\epsilon_{ij})(e_s)=e_s$ for $s \neq i$. Also $\hat{\phi}_G(\epsilon_{ij})(e_s)=t_j^{-1}e_s$ for $s=i$,
\item for $i=n$ and $1 \leq j \leq n$, $\hat{\phi}_G(\epsilon_{nj})(e_s)=e_s$.
\end{itemize}
So $<e_s>$ is an invariant subspace of $\mathbb{C}^{n-1}$ under $\hat{\phi}_G$, hence $\hat{\phi}_G$ is reducible.

For the sufficient condition, suppose that $t_i \neq 1$ for all $1 \leq i < n$.\\
Let $S$ be a non zero invariant subspace of $\mathbb{C}^{n-1}$ under $\hat{\phi}_G$, and let $x=(x_1,x_2,\ldots,x_{n-1})$ be a non zero vector in $S$. Fix $1 \leq r \leq n-1$.
\begin{itemize}
\item If $x_r \neq 0$, then $\hat{\phi}_G(\epsilon_{rn})(x)-x=(t_n^{-1}-1)x_re_r\in S$. But $(t_n^{-1}-1)x_r\neq 0$, so $e_r \in S$.
\item If $x_r=0$, then pick $1\leq j\neq r \leq n-1$ such that $x_j \neq 0$. We have $\hat{\phi}_G(\epsilon_{rj})(x)-x=-t_j^{-1}(t_r-1)x_je_r \in S$ with $-t_j^{-1}(t_r-1)x_j\neq 0$, so $e_r \in S.$
\end{itemize}
Hence $e_r \in S$ for any $1 \leq r \leq n-1$, and so $S=\mathbb{C}^{n-1}$. Thus  $\hat{\phi}_G$ is irreducible. 
\end{proof}

\section{The irreducibility of the representations $\hat{\rho}_B$} 

The group $C_n$ is a semidirect product $C_n= Cb_n \leftthreetimes S_n$. We let $t_1=t_2= \ldots =t_n$ in the matrix $\hat{\rho}_G(\epsilon_{ij})$ in order to get a matrix $\hat{\rho}_B(\epsilon_{ij})$. To each automorphism in $S_n$, we assign the matrix of the corresponding permutation of the elements of the base $W_n$. In this way we obtain the representation $\hat{\rho}_B:C_n \mapsto GL_n(\mathbb{Z}[t^{\pm1}])$ (see [2]).
\begin{definition}
\text{[2]}
The representation $\hat{\rho}_B$ is defined as follows
$$\hat{\rho}_B:C_n \mapsto GL_n(\mathbb{Z}[t^{\pm1}])$$
$$\epsilon_{ij} \mapsto
\left( \begin{array}{c|@{}c|c@{}}
   \begin{matrix}
     I_{i-1} 
   \end{matrix} 
      & 0 & 0 \\
      \hline
    0 &\hspace{0.2cm} \begin{matrix}
   		t^{-1} & 0 & \dots  & \dots& \dots&  1-t^{-1} \\
   		0 & 1  & 0 & \dots & \dots & 0 \\
   		0 & 0 & 1 & 0 & \dots & 0\\
   		\vdots & \vdots & & \ddots & & \vdots \\
   		0 & \dots & \dots & 0 & 1 & 0 \\
   		0 &  0 &\dots & \dots & 0 & 1  \\
\end{matrix}  & 0  \\
\hline
0 & 0 & I_{n-j}
\end{array} \right) \hspace*{0.2cm} for \hspace*{0.2cm} i<j,$$ 

$$\epsilon_{ij} \mapsto
\left( \begin{array}{c|@{}c|c@{}}
   \begin{matrix}
     I_{j-1} 
   \end{matrix} 
      & 0 & 0 \\
      \hline
    0 &\hspace{0.2cm} \begin{matrix}
   		1 & 0 & \dots  & \dots& 0 &  0 \\
   		0 & 1  & 0 & \dots & \dots & 0 \\
   		\vdots & \vdots & & \ddots & & \vdots \\
   		0 & \dots & 0 & 1 & 0 & 0\\
   		0 & \dots & \dots & 0 & 1 & 0 \\
   		1-t^{-1} &  0 &\dots & \dots & 0 & t^{-1}  \\
\end{matrix}  & 0  \\
\hline
0 & 0 & I_{n-i}
\end{array} \right) \hspace*{0.2cm} for \hspace*{0.2cm} j<i,$$ 

$$\alpha_i \mapsto
\left( \begin{array}{c|@{}c|c@{}}
   \begin{matrix}
     I_{i-1} 
   \end{matrix} 
      & 0 & 0 \\
      \hline
    0 &\hspace{0.2cm} \begin{matrix}
   		0 & 1\\
   		1 & 0\\
\end{matrix}  & 0  \\
\hline
0 & 0 & I_{n-i-1}
\end{array} \right) \hspace*{0.2cm} for \hspace*{0.2cm} 1\leq i \leq n-1.$$ 

\end{definition}

Note that for $\epsilon_{ij}$, in the cases $i=1$ and $j=1$, we omit the first $i-1$ (respectively $j-1$) rows and $i-1$ (respectively $j-1$) columns. And in the cases $i=n$ and $j=n$, we omit the last $n-i$ (respectively $n-j$) rows and $n-i$ (respectively $n-j$) columns.

For $\alpha_i$'s, in the case $i=1$, we omit the first $i-1$ rows and $i-1$ columns. And 
in the case $i=n-1$, we omit the last $n-i-1$ rows and $n-i-1$ columns.
\begin{theorem}
The representation $\hat{\rho}_B$ is reducible.
\end{theorem}
\begin{proof}
Let $v=[1,1,...,1]^T$, where $T$ is the transpose. We see that $\epsilon_{ij}(v)=v$ for all $1\leq i,j\leq n$, and $\alpha_i(v)=v$ for all $1\leq i \leq n-1$. So $v$ is fixed under the generators of  $\hat{\rho}_B$. Thus $\hat{\rho}_B$ is reducible.
\end{proof}

We now specialize $t$ to a non-zero complex number and we find a composition factor of degree $n-1$ of $\hat{\rho}_B$. For $\mathbb{C}^n$, consider the basis $S=\{e_1,e_2,...,e_{n-1},v\}$, where $v=[1,1,...,1]^T$.\\
Consider first the action of $\epsilon_{ij}$'s on the basis $S$.\\
\underline{For $i<j \neq n$}: \vspace{0.2cm} \\
$\epsilon_{ij}(e_1)=e_1$, $\epsilon_{ij}(e_2)=e_2$, \dots, $\epsilon_{ij}(e_{i-1})=e_{i-1}$, $\epsilon_{ij}(e_{i})=t^{-1}e_i$, $\epsilon_{ij}(e_{i+1})=e_{i+1}$, \dots, $\epsilon_{ij}(e_{j-1})=e_{j-1}$, $\epsilon_{ij}(e_j)=(1-t^{-1})e_i+e_j$, $\epsilon_{ij}(e_{j+1})=e_{j+1}$, \dots, $\epsilon_{ij}(e_{n-1})=e_{n-1}$, $\epsilon_{ij}(v)=v.$\\
\underline{For $j=n$}: \vspace{0.2cm} \\
$\epsilon_{in}(e_1)=e_1$, $\epsilon_{in}(e_2)=e_2$, \dots, $\epsilon_{in}(e_{i-1})=e_{i-1}$, $\epsilon_{in}(e_{i})=t^{-1}e_i$, $\epsilon_{in}(e_{i+1})=e_{i+1}$, \dots, $\epsilon_{in}(e_{n-1})=e_{n-1}$, \dots, $\epsilon_{in}(v)=v.$\\
\underline{For $j<i \neq n$}:  \vspace{0.2cm} \\
$\epsilon_{ij}(e_1)=e_1$, $\epsilon_{ij}(e_2)=e_2$, \dots, $\epsilon_{ij}(e_{j-1})=e_{j-1}$, $\epsilon_{ij}(e_{j})=e_{j}+(1-t^{-1})e_i$, $\epsilon_{ij}(e_{j+1})=e_{j+1}$, \dots, $\epsilon_{ij}(e_{i-1})=e_{i-1}$, $\epsilon_{ij}(e_{i})=t^{-1}e_i$, $\epsilon_{ij}(e_{i+1})=e_{i+1}$, \dots, $\epsilon_{ij}(e_{n-1})=e_{n-1}$, $\epsilon_{ij}(v)=v.$\\
\underline{For $i=n$}:  \vspace{0.2cm} \\
$\epsilon_{nj}(e_1)=e_1$, $\epsilon_{nj}(e_2)=e_2$, \dots, $\epsilon_{nj}(e_{j-1})=e_{j-1}$, $\epsilon_{nj}(e_j)=(t^{-1}-1)e_1+(t^{-1}-1)e_2+ ... +(t^{-1}-1)e_{j-1}+t^{-1}e_j+(t^{-1}-1)e_{j+1}+ ... +(t^{-1}-1)e_{n-1}+t^{-1}v$, $\epsilon_{nj}(e_{j+1})=e_{j+1}$, \dots, $\epsilon_{nj}(e_{n-1})=e_{n-1}$, $\epsilon_{nj}(v)=v$.
\\

Now, we consider the action of $\alpha_i$'s on the basis $S$.\\
\underline{For $i\neq n-1$}:  \vspace{0.2cm} \\
$\alpha_i(e_1)=e_1$, $\alpha_i(e_2)=e_2$, \dots, $\alpha_i(e_{i-1})=e_{i-1}$, $\alpha_i(e_i)=e_{i+1}$, $\alpha_i(e_{i+1})=e_i$, $\alpha_i(e_{i+2})=e_{i+2}$, \dots, $\alpha_i(e_{n-1})=e_{n-1}$, $\alpha_i(v)=v.$\\
\underline{For $i=n-1$}:  \vspace{0.2cm} \\
$\alpha_{n-1}(e_1)=e_1$, $\alpha_{n-1}(e_2)=e_2$, \dots, $\alpha_{n-1}(e_{n-2})=e_{n-2}$, $\alpha_{n-1}(e_{n-1})=-e_1-e_2-...-e_{n-1}+v$, $\alpha_{n-1}(v)=v.$

\vspace{0.3cm}
So, the representation $\hat{\rho}_B$ in the new basis $S$ becomes as follows
$$\epsilon_{ij} \mapsto
\left( \begin{array}{c|@{}c|c@{}}
   \begin{matrix}
     I_{i-1} 
   \end{matrix} 
      & 0 & 0 \\
      \hline
    0 &\hspace{0.2cm} \begin{matrix}
   		t^{-1} & 0 & \dots  & \dots& \dots& 0 \\
   		0 & 1  & 0 & \dots & \dots & 0 \\
   		0 & 0 & 1 & 0 & \dots & 0\\
   		\vdots & \vdots & & \ddots & & \vdots \\
   		0 & \dots & \dots & 0 & 1 & 0 \\
   		 1-t^{-1} &  0 &\dots & \dots & 0 & 1  \\
\end{matrix}  & 0  \\
\hline
0 & 0 & I_{n-j}
\end{array} \right) \hspace*{0.2cm} for \hspace*{0.2cm} i<j \neq n,$$ 

$$\epsilon_{in} \mapsto
\left( \begin{array}{c|@{}c|c@{}}
   \begin{matrix}
     I_{i-1} 
   \end{matrix} 
      & 0 & 0 \\
      \hline
    0 &\hspace{0.2cm} \begin{matrix}
   		t^{-1} \\
\end{matrix}  & 0  \\
\hline
0 & 0 & I_{n-i}
\end{array} \right),$$

 $$\epsilon_{ij} \mapsto
\left( \begin{array}{c|@{}c|c@{}}
   \begin{matrix}
     I_{j-1} 
   \end{matrix} 
      & 0 & 0 \\
      \hline
    0 &\hspace{0.2cm} \begin{matrix}
   		1 & 0 & \dots  & \dots& 0 &  1-t^{-1} \\
   		0 & 1  & 0 & \dots & \dots & 0 \\
   		\vdots & \vdots & & \ddots & & \vdots \\
   		0 & \dots & \dots & 1 & 0 & 0\\
   		0 & \dots & \dots & 0 & 1 & 0 \\
   		0 &  0 &\dots & \dots & 0 & t^{-1}  \\
\end{matrix}  & 0  \\
\hline
0 & 0 & I_{n-i}
\end{array} \right) \hspace*{0.2cm} for \hspace*{0.2cm} j<i \neq n,$$ 

$$\epsilon_{nj} \mapsto
\left( \begin{array}{c|@{}c@{}}
   \begin{matrix}
     I_{j-1} 
   \end{matrix} 
      & 0 \\
      \hline
    \begin{matrix}
   		t^{-1}-1 & \dots  & t^{-1}-1 \\
   		0 & 0  & 0 \\
   		\vdots  & \vdots  & \vdots \\
   		0 &  0 & 0  \\
   		0 &  0 & 0  \\

\end{matrix}  &\hspace{0.2cm} \begin{matrix}
   		t^{-1} & t^{-1}-1  & \dots  &  t^{-1}-1 & t^{-1}\\
   		0 & 1  & 0 & \dots & 0 \\
   		\vdots & & \ddots  & &\vdots \\
   		0 & \dots & 0  & 1 & 0 \\
   		0 & \dots & \dots  & 0 & 1 \\
\end{matrix} \\
\end{array} \right),$$ 

$$\alpha_i \mapsto
\left( \begin{array}{c|@{}c|c@{}}
   \begin{matrix}
     I_{i-1} 
   \end{matrix} 
      & 0 & 0 \\
      \hline
    0 &\hspace{0.2cm} \begin{matrix}
   		0 & 1\\
   		1 & 0\\
\end{matrix}  & 0  \\
\hline
0 & 0 & I_{n-i-1}
\end{array} \right) \hspace*{0.2cm} for \hspace*{0.2cm} 1\leq i< n-1,$$ 

$$\alpha_{n-1} \mapsto
\left( \begin{array}{c|@{}c@{}}
   \begin{matrix}
     I_{n-2} 
   \end{matrix} 
      & 0 \\
      \hline
    \begin{matrix}
   		-1 & -1 & \dots & -1 \\
   		0 & 0  & 0 & 0\\

\end{matrix}  &\hspace{0.2cm} \begin{matrix}
   	-1& 1 \\
   		0 & 1\\
\end{matrix} \\
\end{array} \right).$$ 
\vspace{0.3cm}

Now, we remove the last row and the last column to obtain the $n-1$ composition factor $\hat{\phi}_B$, which is given by
$$\epsilon_{ij} \mapsto
\left( \begin{array}{c|@{}c|c@{}}
   \begin{matrix}
     I_{i-1} 
   \end{matrix} 
      & 0 & 0 \\
      \hline
    0 &\hspace{0.2cm} \begin{matrix}
   		t^{-1} & 0 & \dots  & \dots& \dots& 0 \\
   		0 & 1  & 0 & \dots & \dots & 0 \\
   		0 & 0 & 1 & 0 & \dots & 0\\
   		\vdots & \vdots & & \ddots & & \vdots \\
   		0 & \dots & \dots & \dots & 1 & 0 \\
   		 1-t^{-1} &  0 &\dots & \dots & 0 & 1  \\
\end{matrix}  & 0  \\
\hline
0 & 0 & I_{n-j-1}
\end{array} \right) \hspace*{0.2cm} for \hspace*{0.2cm} i<j \neq n,$$ 

$$\epsilon_{in} \mapsto
\left( \begin{array}{c|@{}c|c@{}}
   \begin{matrix}
     I_{i-1} 
   \end{matrix} 
      & 0 & 0 \\
      \hline
    0 &\hspace{0.2cm} \begin{matrix}
   		t^{-1} \\
\end{matrix}  & 0  \\
\hline
0 & 0 & I_{n-i-1}
\end{array} \right),$$

 $$\epsilon_{ij} \mapsto
\left( \begin{array}{c|@{}c|c@{}}
   \begin{matrix}
     I_{j-1} 
   \end{matrix} 
      & 0 & 0 \\
      \hline
    0 &\hspace{0.2cm} \begin{matrix}
   		1 & 0 & \dots  & \dots& 0 &  1-t^{-1} \\
   		0 & 1  & 0 & \dots & \dots & 0 \\
   		\vdots & \vdots & & \ddots & & \vdots \\
   		0 & \dots & \dots & 1 & 0 & 0\\
   		0 & \dots & \dots & 0 & 1 & 0 \\
   		0 &  0 &\dots & \dots & 0 & t^{-1}  \\
\end{matrix}  & 0  \\
\hline
0 & 0 & I_{n-i-1}
\end{array} \right) \hspace*{0.2cm} for \hspace*{0.2cm} j<i \neq n,$$ 

$$\epsilon_{nj} \mapsto
\left( \begin{array}{c|@{}c@{}}
   \begin{matrix}
     I_{j-1} 
   \end{matrix} 
      & 0 \\
      \hline
    \begin{matrix}
   		t^{-1}-1 & \dots  & t^{-1}-1 \\
   		0 & 0  & 0 \\
   		\vdots  & \vdots  & \vdots \\
   		0 &  0 & 0  \\
   		0 &  0 & 0  \\

\end{matrix}  &\hspace{0.2cm} \begin{matrix}
   		t^{-1} & t^{-1}-1  & \dots  & \dots & t^{-1}-1\\
   		0 & 1  & 0 & \dots & 0 \\
   		\vdots & & \ddots  & &\vdots \\
   		0 & \dots & 0  & 1 & 0 \\
   		0 & \dots & \dots  & 0 & 1 \\
\end{matrix} \\
\end{array} \right),$$ 

$$\alpha_i \mapsto
\left( \begin{array}{c|@{}c|c@{}}
   \begin{matrix}
     I_{i-1} 
   \end{matrix} 
      & 0 & 0 \\
      \hline
    0 &\hspace{0.2cm} \begin{matrix}
   		0 & 1\\
   		1 & 0\\
\end{matrix}  & 0  \\
\hline
0 & 0 & I_{n-i-2}
\end{array} \right) \hspace*{0.2cm} for \hspace*{0.2cm} 1\leq i< n-1,$$ 

$$\alpha_{n-1} \mapsto
\left( \begin{array}{c|@{}c@{}}
   \begin{matrix}
     I_{n-2} 
   \end{matrix} 
      & 0 \\
      \hline
    \begin{matrix}
   		-1 & -1 & \dots & -1 \\

\end{matrix}  &\hspace{0.2cm} \begin{matrix}
   	-1 \\
\end{matrix} \\
\end{array} \right).$$
\\

Now, we consider the complex specialization of $\hat{\phi}_B$ by letting $t$ be a non-zero complex number.
\begin{theorem}
Let $0\neq t \in \mathbb{C}$. The representation $\hat{\phi}_B(t): C_{n} \mapsto GL_{n-1}(\mathbb{C})$ is irreducible.
\end{theorem}
\begin{proof}
Since the restriction of the representation $\hat{\phi}_B(t)$ to the subgroup $S_n$ inside $C_n$ is irreducible, it follows that $\hat{\phi}_B(t)$ itself is irreducible.
\end{proof}

\section{The tensor product of complex irreducible representations of $Cb_3$}   

In this section, we set $n=3$ and we consider the complex irreducible specialization $\hat{\phi}_G$, which is given by
$$\epsilon_{12} \mapsto
\left( \begin{array}{c@{}c@{}}
\begin{matrix}
 t_2^{-1} & 0 \\
 t_2^{-1}(t_1-1) & 1 \\
\end{matrix} 
\end{array} \right), \epsilon_{21} \mapsto
\left( \begin{array}{c@{}c@{}}
\begin{matrix}
 1 & t_1^{-1}(t_2-1) \\
 0 & t_1^{-1} \\
\end{matrix} 
\end{array} \right), \epsilon_{13} \mapsto
\left( \begin{array}{c@{}c@{}}
\begin{matrix}
 t_3^{-1} & 0 \\
 0 & 1 \\
\end{matrix} 
\end{array} \right),$$

$$ \epsilon_{31} \mapsto
\left( \begin{array}{c@{}c@{}}
\begin{matrix}
 t_1^{-1} & -t_1^{-1}(t_2-1) \\
 0 & 1 \\
\end{matrix} 
\end{array} \right), \epsilon_{32} \mapsto
\left( \begin{array}{c@{}c@{}}
\begin{matrix}
 1 & 0 \\
 -t_2^{-1}(t_1-1) & t_2^{-1} \\
\end{matrix} 
\end{array} \right), \epsilon_{23} \mapsto
\left( \begin{array}{c@{}c@{}}
\begin{matrix}
 1 & 0 \\
 0 & t_3^{-1} \\
\end{matrix} 
\end{array} \right).$$\\
Now, we consider the generators of $\hat{\phi}_{G}(t_1,t_2,t_3) \otimes \hat{\phi}_{G}(m_1,m_2,m_3)$. For simplicity, we write $(\hat{\phi}_{G}(t_1,t_2,t_3) \otimes \hat{\phi}_{G}(m_1,m_2,m_3))(\epsilon_{ij})=\epsilon_{ij}$.
$$\epsilon_{12} \mapsto
\left( \begin{array}{c@{}c@{}}
\begin{matrix}
 t_2^{-1}m_2^{-1} & 0  & 0 & 0\\
 t^{-1}_2(t_1-1)m_2^{-1} & m_2^{-1} & 0 & 0 \\
 t^{-1}_2m_2^{-1}(m_1-1) & 0 & t_2^{-1} & 0 \\
 t_2^{-1}(t_1-1)m_2^{-1}(m_1-1) & m_2^{-1}(m_1-1) & t_2^{-1}(t_1-1) & 1 \\
\end{matrix} 
\end{array} \right),$$

$$\epsilon_{21} \mapsto
\left( \begin{array}{c@{}c@{}}
\begin{matrix}
 1 & t^{-1}_1(t_2-1)  & m^{-1}_1(m_2-1) & t_1^{-1}(t_2-1)m_1^{-1}(m_2-1)\\
 0 & t_1^{-1} & 0 & t_1^{-1}m_1^{-1}(m_2-1) \\
 0 & 0 & m_1^{-1} & t_1^{-1}m_1^{-1}(t_2-1) \\
 0 & 0 & 0 & t_1^{-1}m_1^{-1} \\
\end{matrix} 
\end{array} \right),$$

$$\epsilon_{13} \mapsto
\left( \begin{array}{c@{}c@{}}
\begin{matrix}
 t_3^{-1}m_3^{-1} & 0  & 0 & 0\\
 0 & m_3^{-1} & 0 & 0 \\
 0 & 0 & t_3^{-1} & 0 \\
 0 & 0 & 0 & 1 \\
\end{matrix} 
\end{array} \right),$$

$$\epsilon_{31} \mapsto
\left( \begin{array}{c@{}c@{}}
\begin{matrix}
 t_1^{-1}m_1^{-1} & -t^{-1}_1(t_2-1)m^{-1}_1 & -t^{-1}_1m_1^{-1}(m_2-1) & t_1^{-1}(t_2-1)m_1^{-1}(m_2-1)\\
 0 & m_1^{-1} & 0 & -m_1^{-1}(m_2-1) \\
 0 & 0 & t_1^{-1} & -t_1^{-1}(t_2-1) \\
 0 & 0 & 0 & 1 \\
\end{matrix} 
\end{array} \right),$$

$$\epsilon_{23} \mapsto
\left( \begin{array}{c@{}c@{}}
\begin{matrix}
 1 & 0 & 0 & 0\\
 0 & t_3^{-1} & 0 & 0 \\
 0 & 0 & m_3^{-1} & 0 \\
 0 & 0 & 0 & t_3^{-1}m_3^{-1}\\
\end{matrix} 
\end{array} \right),$$

$$\epsilon_{32} \mapsto
\left( \begin{array}{c@{}c@{}}
\begin{matrix}
1 & 0 & 0 & 0 \\
-t_2^{-1}(t_1-1) & t_2^{-1} & 0 & 0 \\
-m_2^{-1}(m_1-1) & 0 & m_2^{-1} & 0 \\
t_2^{-1}(t_1-1)m_2^{-1}(m_1-1) & -t^{-1}_2m_2^{-1}(m_1-1) & -t_2^{-1}(t_1-1)m_2^{-1} & t_2^{-1}m_2^{-1} \\
\end{matrix} 
\end{array} \right).$$

By Theorem 4, we assume that $t_i\neq 1$ for all $1 \leq i\leq 3$ and $m_i\neq 1$ for all $1 \leq i\leq 3$.

\begin{theorem}
For $n=3$, the tensor product representation  $\hat{\phi}_G(t_1,t_2,t_3) \otimes \hat{\phi}_G(m_1,m_2,m_3)$ is irreducible if and only if $(t_1,t_2,t_3)$ and $(m_1,m_2,m_3)$ are distinct vectors.
\end{theorem}
\begin{proof} For the necessary condition, suppose that $(t_1,t_2,t_3)$ and $(m_1,m_2,m_3)$ are equal vectors. Consider $S_1=<e_1,e_2+e_3,e_4>$.\\
$\epsilon_{12}(e_1)=t_2^{-2}e_1+t_2^{-2}(t_1-1)(e_2+e_3)+t_2^{-2}(t_1-1)^2e_4 \in S_1 \\
\epsilon_{21}(e_1)=e_1 \in S_1 \\
\epsilon_{13}(e_1)=t^{-2}_3e_1 \in S_1\\
\epsilon_{31}(e_1)=t_1^{-2}e_1 \in S_1 \\
\epsilon_{23}(e_1)=e_1 \in S_1\\
\epsilon_{32}(e_1)=e_1-t_2^{-1}(t_1-1)(e_2+e_3)+(t_1-1)^2t_2^{-2}e_4 \in S_1\\
\epsilon_{12}(e_2+e_3)=t_2^{-1}(e_2+e_3)+2t_2^{-1}(t_1-1)e_4 \in S_1\\
\epsilon_{21}(e_2+e_3)=2t_1^{-1}(t_2-1)e_1+t_1^{-1}(e_2+e_3) \in S_1\\ 
\epsilon_{13}(e_2+e_3)=t_3^{-1}(e_2+e_3) \in S_1\\
\epsilon_{31}(e_2+e_3)=-2t_1^{-1}(t_2-1)e_1+t_1^{-1}(e_2+e_3) \in S_1\\
\epsilon_{23}(e_2+e_3)=t_3^{-1}(e_2+e_3) \in S_1\\
\epsilon_{32}(e_2+e_3)=t_2^{-1}(e_2+e_3)-t^{-2}_2(t_1-1)^2e_4 \in S_1\\ 
\epsilon_{12}(e_4)=e_4 \in S_1\\ 
\epsilon_{21}(e_4)=t^{-2}_1(t_2-1)^2e_1+t_1^{-2}(t_2-1)(e_2+e_3)+t_1^{-2}e_4 \in S_1\\ 
\epsilon_{13}(e_4)=e_4 \in S_1\\ 
\epsilon_{31}(e_4)=t_1^{-2}(t_2-1)^2e_1-t_1^{-1}(t_2-1)(e_2+e_3)+e_4 \in S_1\\ 
\epsilon_{23}(e_4)=t_3^{-2}e_4 \in S_1\\ 
\epsilon_{32}(e_4)=t_2^{-2}e_4 \in S_1\\ $
\\
Therefore, $S_1$ is a non trivial invariant subspace of $\mathbb{C}^4$ under $\hat{\phi}_G(t_1,t_2,t_3) \otimes \hat{\phi}_B(m_1,m_2,m_3)$. Hence $\hat{\phi}_G(t_1,t_2,t_3) \otimes \hat{\phi}_B(m_1,m_2,m_3)$ is reducible.

For the sufficient condition, suppose that the vectors $(t_1,t_2,t_3)$ and $(m_1,m_2,m_3)$ are distinct. Let $S$ be a non trivial invariant subspace of $\mathbb{C}^4$ under $\hat{\phi}_G(t_1,t_2,t_3) \otimes \hat{\phi}_B(m_1,m_2,m_3)$.
\begin{itemize}
\item[(a)] Suppose $t_3\neq m_3$ and $t_3m_3\neq 1$.\\
In this case, the diagonal matrix $\epsilon_{13}$ has distinct eigenvalues, so $S=<e_i>$ or $S=<e_i,e_j>$ or $S=<e_i,e_j,e_k>$, where $1 \leq i,j,k \leq 4$.
\begin{itemize}
\item[(i)] $S\neq <e_i>$ for all $1\leq i \leq 4$.\\
If $S=<e_1>$, then $\epsilon_{12}(e_1)=\beta_1e_1+\beta_2e_2+\beta_3e_3+\beta_4e_4$ with $\beta_2=t_2^{-1}(t_1-1)m_2^{-1} \neq 0$, which is a contradiction.\\
If $S=<e_2>$, then $\epsilon_{12}(e_2)=\beta_2e_2+\beta_4e_4$ with $\beta_4=m_2^{-1}(m_1-1) \neq 0$, which is a contradiction.\\
If $S=<e_3>$, then $\epsilon_{12}(e_3)=\beta_3e_3+\beta_4e_4$ with $\beta_4=t_2^{-1}(t_1-1) \neq 0$, which is a contradiction.\\
If $S=<e_4>$, then $\epsilon_{21}(e_4)=\beta_1e_1+\beta_2e_2+\beta_3e_3+\beta_4e_4$ with $\beta_1=t_1^{-1}(t_2-1)m_1^{-1}(m_2-1) \neq 0$, which is a contradiction.
\item[(ii)] $S\neq <e_i,e_j>$ for all $1\leq i,j \leq 4$.\\
If $S=<e_1,e_2>$, then $\epsilon_{12}(e_1)=\beta_1e_1+\beta_2e_2+\beta_3e_3+\beta_4e_4$ with $\beta_3=t_2^{-1}m_2^{-1}(m_1-1) \neq 0$, which is a contradiction.\\
If $S=<e_1,e_3>$, then $\epsilon_{12}(e_1)=\beta_1e_1+\beta_2e_2+\beta_3e_3+\beta_4e_4$ with $\beta_2=t_2^{-1}(t_1-1)m_2^{-1} \neq 0$, which is a contradiction.\\
If $S=<e_1,e_4>$, then $\epsilon_{12}(e_1)=\beta_1e_1+\beta_2e_2+\beta_3e_3+\beta_4e_4$ with $\beta_2=t_2^{-1}(t_1-1)m_2^{-1}  \neq 0$, which is a contradiction.\\
If $S=<e_2,e_3>$, then $\epsilon_{12}(e_2)=\beta_2e_2+\beta_4e_4$ with $\beta_4=m_2^{-1}(m_1-1) \neq 0$, which is a contradiction.\\
If $S=<e_2,e_4>$, then $\epsilon_{31}(e_2)=\beta_1e_1+\beta_2e_2$ with $\beta_1=-t^{-1}_1(t_2-1)m_1^{-1} \neq 0$, which is a contradiction.\\
If $S=<e_3,e_4>$, then $\epsilon_{31}(e_3)=\beta_1e_1+\beta_3e_3$ with $\beta_1=-t_1^{-1}m_1^{-1}(m_2-1) \neq 0$, which is a contradiction.
\item[(iii)] $S\neq <e_i,e_j,e_k>$ for all $1\leq i,j,k \leq 4$.\\
If $S=<e_1,e_2,e_3>$, then $\epsilon_{12}(e_1)=\beta_1e_1+\beta_2e_2+\beta_3e_3+\beta_4e_4$ with $\beta_4=t_2^{-1}(t_1-1)m_2^{-1}(m_1-1) \neq 0$, which is a contradiction.\\
If $S=<e_1,e_2,e_4>$, then $\epsilon_{12}(e_1)=\beta_1e_1+\beta_2e_2+\beta_3e_3+\beta_4e_4$ with $\beta_3=t_2^{-1}m_2^{-1}(m_1-1) \neq 0$, which is a contradiction.\\
If $S=<e_1,e_3,e_4>$, then $\epsilon_{12}(e_1)=\beta_1e_1+\beta_2e_2+\beta_3e_3+\beta_4e_4$ with $\beta_2=t_2^{-1}(t_1-1)m_2^{-1} \neq 0$, which is a contradiction.\\
If $S=<e_2,e_3,e_4>$, then $\epsilon_{31}(e_4)=\beta_1e_1+\beta_2e_2+\beta_3e_3+\beta_4e_4$ with $\beta_1=t_1^{-1}(t_2-1)m_1^{-1}(m_2-1) \neq 0$, which is a contradiction.
\end{itemize}
\item[(b)] Suppose $t_3\neq m_3$ and $t_3m_3=1$.\\
In addition to the subspaces mentioned in (a), we have other possible candidates to invariant subspaces. More precisely, we consider $S=<a_1e_1+a_4e_4>$ or $S=<e_2,a_1e_1+a_4e_4>$ or $S=<e_3,a_1e_1+a_4e_4>$ or $S=<e_2,e_3,a_1e_1+a_4e_4>$, $a_1 \neq 0$ and $a_4 \neq 0$.
\begin{itemize}
\item[(i)] $S\neq<a_1e_1+a_4e_4>.$\\
If $S=<a_1e_1+a_4e_4>$, then $\epsilon_{12}(a_1e_1+a_4e_4)=\beta_1e_1+\beta_2e_2+\beta_3e_3+\beta_4e_4$ with $\beta_2=a_1t_2^{-1}(t_1-1)m_2^{-1} \neq 0$, which is a contradiction.
\item[(ii)] $S\neq <e_2,a_1e_1+a_4e_4>$.\\
If $S=<e_2,a_1e_1+a_4e_4>$, then $\epsilon_{12}(e_2)=\beta_2e_2+\beta_4e_4$ with $\beta_4=m_2^{-1}(m_1-1) \neq 0$, and so $e_4 \in S$. This gives a contradiction.
\item[(iii)] $S\neq <e_3,a_1e_1+a_4e_4>$.\\
If $S=<e_3,a_1e_1+a_4e_4>$, then $\epsilon_{12}(e_3)=\beta_3e_3+\beta_4e_4$ with $\beta_4=t_2^{-1}(t_1-1) \neq 0$, and so $e_4 \in S$. This gives a contradiction.
\item[(iv)]  $S\neq <e_2,e_3,a_1e_1+a_4e_4>$.\\
If $S=<e_2,e_3,a_1e_1+a_4e_4>$, then $\epsilon_{12}(e_2)=\beta_2e_2+\beta_4e_4$ with $\beta_4=m_2^{-1}(m_1-1) \neq 0$, and so $e_4 \in S$. This gives a contradiction.
\end{itemize}
\item[(c)] Suppose $t_3=m_3$ and $t_3m_3\neq 1$, then $t_1\neq m_1$ or $t_2\neq m_2$.\\
In addition to the subspaces mentioned in (a), we also have other possible invariant subspaces. For instance, we consider $S=<a_2e_2+a_3e_3>$ or $S=<e_1,a_2e_2+a_3e_3>$ or $S=<e_4,a_2e_2+a_3e_3>$ or $S=<e_1,e_4,a_2e_2+a_3e_3>$, $a_2\neq 0$ and $a_3\neq 0$.
\begin{itemize}
\item[(i)]$S\neq <a_2e_2+a_3e_3>$.\\
If $S=<a_2e_2+a_3e_3>$, then $\epsilon_{12}(a_2e_2+a_3e_3)=m_2^{-1}a_2e_2+t_2^{-1}a_3e_3+\gamma_4e_4\in S$ and $\epsilon_{31}(a_2e_2+a_3e_3)=m_1^{-1}a_2e_2+t_1^{-1}a_3e_3+\delta_4e_4 \in S$, where $\gamma_4$ and $\delta_4$ are scalars. So $\epsilon_{12}(a_2e_2+a_3e_3)= \lambda_1(a_2e_2+a_3e_3)$ and $\epsilon_{31}(a_2e_2+a_3e_3)=\lambda_2(a_2e_2+a_3e_3)$. Thus we get $\lambda_1a_2=m_2^{-1}a_2$ and $\lambda_1a_3=t_2^{-1}a_3$, which implies that $\lambda_1=t^{-1}_2=m^{-1}_2$. Similarly, we have $\lambda_2a_2=m_1^{-1}a_2$ and $\lambda_2a_3=t_1^{-1}a_3$, which implies that $\lambda_2=t^{-1}_1=m^{-1}_1$. This contradicts the fact that $t_1\neq m_1$ or $t_2\neq m_2$.\\
\item[(ii)]$S\neq <e_1,a_2e_2+a_3e_3>$.\\
If $S=<e_1,a_2e_2+a_3e_3>$, then $\epsilon_{12}(e_1)=\beta_1e_1+\beta_2e_2+\beta_3e_3+\beta_4e_4$ with $\beta_4=t_2^{-1}(t_1-1)m_2^{-1}(m_1-1) \neq 0$, which is a contradiction.
\item[(iii)]$S\neq <e_4,a_2e_2+a_3e_3>$.\\
If $S=<e_4,a_2e_2+a_3e_3>$, then $\epsilon_{21}(e_4)=\beta_1e_1+\beta_2e_2+\beta_3e_3+\beta_4e_4$ with $\beta_1=t_1^{-1}(t_2-1)m_1^{-1}(m_2-1) \neq 0$, which is a contradiction.
\item[(iv)]$S\neq <e_1,e_4,a_2e_2+a_3e_3>$.\\
If $S=<e_1,e_4,a_2e_2+a_3e_3>$, then $\epsilon_{12}(a_2e_2+a_3e_3)=m_2^{-1}a_2e_2+t_2^{-1}a_3e_3+\gamma_4e_4\in S$ and $\epsilon_{31}(a_2e_2+a_3e_3)=m_1^{-1}a_2e_2+t_1^{-1}a_3e_3+\delta_4e_4 \in S$, where $\gamma_4$ and $\delta_4$ are scalars. So $\epsilon_{12}(a_2e_2+a_3e_3)= \lambda_1(a_2e_2+a_3e_3)+ \gamma_1e_1+ \omega_1e_4$ and $\epsilon_{31}(a_2e_2+a_3e_3)=\lambda_2(a_2e_2+a_3e_3)+ \gamma_2e_1+ \omega_2e_4$, where $\lambda_1$, $\lambda_2$, $\gamma_1$, $\gamma_2$, $\omega_1$, and $\omega_2$ are non-zero scalars. Thus we get $\lambda_1a_2=m_2^{-1}a_2$ and $\lambda_1a_3=t_2^{-1}a_3$, which implies that $\lambda_1=t^{-1}_2=m^{-1}_2$. Similarly, we have $\lambda_2a_2=m_1^{-1}a_2$ and $\lambda_2a_3=t_1^{-1}a_3$, which implies that $\lambda_2=t^{-1}_1=m^{-1}_1$. This contradicts the fact that $t_1\neq m_1$ or $t_2\neq m_2$.
\item[(d)] Suppose $t_3=m_3$ and $t_3m_3=1$, then $t_1\neq m_1$ or $t_2\neq m_2$.\\
In addition to all previous subspaces mentioned in (a), (b) and (c), we may also consider $S=<a_1e_1+a_4e_4,a_2e_2+a_3e_3>$ with $a_i \neq 0$ for all $1 \leq i \leq 4$.\\
If $S=<a_1e_1+a_4e_4,a_2e_2+a_3e_3>$, then $\epsilon_{12}(a_2e_2+a_3e_3)=m_2^{-1}a_2e_2+t_2^{-1}a_3e_3+\gamma_4e_4 \in S$ and $\epsilon_{31}(a_2e_2+a_3e_3)=m_1^{-1}a_2e_2+t_1^{-1}a_3e_3+\delta_4e_4 \in S$, where $\gamma_4$ and $\delta_4$ are non-zero scalars. So $\epsilon_{12}(a_2e_2+a_3e_3)= \lambda_1(a_2e_2+a_3e_3)+\gamma_1(a_1e_1+a_4e_4)$ and $\epsilon_{31}(a_2e_2+a_3e_3)=\lambda_2(a_2e_2+a_3e_3)+\gamma_2(a_1e_1+a_4e_4)$, where $\lambda_1$, $\lambda_2$, $\gamma_1$, and $\gamma_2$ are non-zero scalars. So $\lambda_1a_2=m_2^{-1}a_2$ and $\lambda_1a_3=t_2^{-1}a_3$, which implies that $\lambda_1=t^{-1}_2=m^{-1}_2$. Similarly, $\lambda_2a_2=m_1^{-1}a_2$ and $\lambda_2a_3=t_1^{-1}a_3$, which implies that $\lambda_2=t^{-1}_1=m^{-1}_1$. This contradicts the fact that $t_1\neq m_1$ or $t_2\neq m_2$.
\end{itemize}
\end{itemize}
Thus, $\mathbb{C}^4$ has no non trivial invariant subspace under $\hat{\phi}_G(t_1,t_2,t_3) \otimes \hat{\phi}_B(m_1,m_2,m_3)$. Therefore $\hat{\phi}_G(t_1,t_2,t_3) \otimes \hat{\phi}_B(m_1,m_2,m_3)$ is irreducible.
\end{proof}

\section{The tensor product of complex irreducible representations of $C_3$}   

In this section, we set $n=3$ and we consider the irreducible complex specialization $\hat{\phi}_B$, which is given by
$$\epsilon_{12} \mapsto
\left( \begin{array}{c@{}c@{}}
\begin{matrix}
 t^{-1} & 0 \\
 1-t^{-1} & 1 \\
\end{matrix} 
\end{array} \right), \epsilon_{21} \mapsto
\left( \begin{array}{c@{}c@{}}
\begin{matrix}
 1 & 1-t^{-1} \\
 0 & t^{-1} \\
\end{matrix} 
\end{array} \right), \epsilon_{13} \mapsto
\left( \begin{array}{c@{}c@{}}
\begin{matrix}
 t^{-1} & 0 \\
 0 & 1 \\
\end{matrix} 
\end{array} \right),$$

$$ \epsilon_{31} \mapsto
\left( \begin{array}{c@{}c@{}}
\begin{matrix}
 t^{-1} & t^{-1}-1 \\
 0 & 1 \\
\end{matrix} 
\end{array} \right), \epsilon_{32} \mapsto
\left( \begin{array}{c@{}c@{}}
\begin{matrix}
 1 & 0 \\
 t^{-1}-1 & t^{-1} \\
\end{matrix} 
\end{array} \right), \epsilon_{23} \mapsto
\left( \begin{array}{c@{}c@{}}
\begin{matrix}
 1 & 0 \\
 0 & t^{-1} \\
\end{matrix} 
\end{array} \right),$$

$$\alpha_{1} \mapsto
\left( \begin{array}{c@{}c@{}}
\begin{matrix}
 0 & 1 \\
 1 & 0 \\
\end{matrix} 
\end{array} \right), \alpha_{2} \mapsto
\left( \begin{array}{c@{}c@{}}
\begin{matrix}
 1 & 0 \\
 -1 & -1 \\
\end{matrix} 
\end{array} \right).
$$

Now, we consider the generators of $\hat{\phi}_B(t)\otimes \hat{\phi}_B(m)$. For simplicity, we set $(\hat{\phi}_B(t)\otimes \hat{\phi}_B(m))(\epsilon_{ij})=\epsilon_{ij}$ and $(\hat{\phi}_B(t)\otimes \hat{\phi}_B(m))(\alpha_i)=\alpha_i$.
$$\epsilon_{12} \mapsto
\left( \begin{array}{c@{}c@{}}
\begin{matrix}
 t^{-1}m^{-1} & 0  & 0 & 0\\
 t^{-1}(1-m^{-1}) & t^{-1} & 0 & 0 \\
 m^{-1}(1-t^{-1}) & 0 & m^{-1} & 0 \\
 (1-t^{-1})(1-m^{-1}) & (1-t^{-1}) & (1-m^{-1}) & 1 \\
\end{matrix} 
\end{array} \right),$$

$$\epsilon_{21} \mapsto
\left( \begin{array}{c@{}c@{}}
\begin{matrix}
 1 & 1-m^{-1}  & 1-t^{-1} & (1-t^{-1})(1-m^{-1})\\
 0 & m^{-1} & 0 & m^{-1}(1-t^{-1}) \\
 0 & 0 & t^{-1} & t^{-1}(1-m^{-1}) \\
 0 & 0 & 0 & t^{-1}m^{-1} \\
\end{matrix} 
\end{array} \right),$$

$$\epsilon_{13} \mapsto
\left( \begin{array}{c@{}c@{}}
\begin{matrix}
 t^{-1}m^{-1} & 0  & 0 & 0\\
 0 & t^{-1} & 0 & 0 \\
 0 & 0 & m^{-1} & 0 \\
 0 & 0 & 0 & 1 \\
\end{matrix} 
\end{array} \right),$$

$$\epsilon_{31}\mapsto
\left( \begin{array}{c@{}c@{}}
\begin{matrix}
 t^{-1}m^{-1} & t^{-1}(m^{-1}-1) & m^{-1}(t^{-1}-1) & (t^{-1}-1)(m^{-1}-1)\\
 0 & t^{-1} & 0 & t^{-1}-1 \\
 0 & 0 & m^{-1} & m^{-1}-1 \\
 0 & 0 & 0 & 1 \\
\end{matrix} 
\end{array} \right),$$

$$\epsilon_{23}\mapsto
\left( \begin{array}{c@{}c@{}}
\begin{matrix}
 1 & 0 & 0 & 0\\
 0 & m^{-1} & 0 & 0 \\
 0 & 0 & t^{-1} & 0 \\
 0 & 0 & 0 & t^{-1}m^{-1}\\
\end{matrix} 
\end{array} \right),$$

$$\epsilon_{32}\mapsto
\left( \begin{array}{c@{}c@{}}
\begin{matrix}
1 & 0 & 0 & 0 \\
m^{-1}-1 & m^{-1} & 0 & 0 \\
t^{-1}-1 & 0 & t^{-1} & 0 \\
(t^{-1}-1)(m^{-1}-1) & m^{-1}(t^{-1}-1) & t^{-1}(m^{-1}-1) & t^{-1}m^{-1} \\
\end{matrix} 
\end{array} \right),$$

$$\alpha_{1}\mapsto
\left( \begin{array}{c@{}c@{}}
\begin{matrix}
0 & 0 & 0 & 1 \\
0 & 0 & 1 & 0 \\
0 & 1 & 0 & 0 \\
1 & 0 & 0 & 0 \\
\end{matrix} 
\end{array} \right),$$

$$\alpha_{2}\mapsto
\left( \begin{array}{c@{}c@{}}
\begin{matrix}
1 & 0 & 0 & 0 \\
-1 & -1 & 0 & 0 \\
-1 & 0 & -1 & 0 \\
1 & 1 & 1 & 1 \\
\end{matrix} 
\end{array} \right).$$

\begin{theorem}
For $n=3$, the tensor product representation  $\hat{\phi}_B(t) \otimes \hat{\phi}_B(m)$ is irreducible if and only if $t \neq m$.
\end{theorem}
\begin{proof} For the necessary condition, we suppose that $t=m$, and we consider $S_1=\{e_1,e_2+e_3,e_4\}$.\\
$\epsilon_{12}(e_1)=t^{-2}e_1+t^{-1}(1-t^{-1})(e_2+e_3)+(1-t^{-1})^2e_4 \in S_1 \\
\epsilon_{21}(e_1)=e_1 \in S_1 \\
\epsilon_{13}(e_1)=e_1 \in S_1\\
\epsilon_{31}(e_1)=t^{-2}e_1 \in S_1 \\
\epsilon_{23}(e_1)=e_1 \in S_1\\
\epsilon_{32}(e_1)=e_1+(t^{-1}-1)(e_2+e_3)+(t^{-1}-1)^2e_4 \in S_1\\
\epsilon_{12}(e_2+e_3)=t^{-1}(e_2+e_3)+2(1-t^{-1})e_4 \in S_1\\
\epsilon_{21}(e_2+e_3)=2(1-t^{-1})e_1+t^{-1}(e_2+e_3) \in S_1\\ 
\epsilon_{13}(e_2+e_3)=t^{-1}(e_2+e_3) \in S_1\\
\epsilon_{31}(e_2+e_3)=2t^{-1}(1-t^{-1})e_1+t^{-1}(e_2+e_3) \in S_1\\
\epsilon_{23}(e_2+e_3)=t^{-1}(e_2+e_3) \in S_1\\
\epsilon_{32}(e_2+e_3)=t^{-1}(e_2+e_3)+t^{-1}(t^{-1}-1)e_4 \in S_1\\ 
\epsilon_{12}(e_4)=e_4 \in S_1\\ 
\epsilon_{21}(e_4)=(1-t^{-1})^2e_1+t^{-1}(1-t^{-1})(e_2+e_3)+t^{-2}e_4 \in S_1\\ 
\epsilon_{13}(e_4)=t^{-2}e_4 \in S_1\\ 
\epsilon_{31}(e_4)=(t^{-1}-1)^2e_1+(t^{-1}-1)(e_2+e_3)+e_4 \in S_1\\ 
\epsilon_{23}(e_4)=t^{-2}e_4 \in S_1\\ 
\epsilon_{32}(e_4)=t^{-2}e_4 \in S_1\\ $
Therefore, $S_1$ is a non trivial invariant subspace of $\mathbb{C}^4$ under $\hat{\phi}_B(t) \otimes \hat{\phi}_B(m)$, and so $\hat{\phi}_B(t) \otimes \hat{\phi}_B(m)$ is reducible.

For the sufficient condition, we suppose $t \neq m$, $tm \neq 1$, $t\neq 1,$ and $m\neq 1$, then we get that $\epsilon_{23}$ has distinct eigenvalues. Suppose $S$ is a non trivial invariant subspace of $\mathbb{C}^4$ under $\hat{\phi}_B(t) \otimes \hat{\phi}_B(m)$, then $S=<e_i>$ or $S=<e_i,e_j>$ or $S=<e_i,e_j,e_k>$, where $1\leq i,j,k\leq 4$.
\begin{itemize}
\item[(a)] $S \neq <e_i>$ for all $1\leq i\leq 4$.\\
If $S=<e_1>$, then $\alpha_1(e_1)=e_4 \notin S$, which is a contradiction.\\
If $S=<e_2>$, then $\alpha_1(e_2)=e_3 \notin S$, which is a contradiction.\\
If $S=<e_3>$, then $\alpha_1(e_3)=e_2 \notin S$, which is a contradiction.\\
If $S=<e_4>$, then $\alpha_1(e_4)=e_1 \notin S$, which is a contradiction.
\item[(b)] $S \neq <e_i,e_j>$ for all $1\leq i,j\leq 4$.\\
If $S=<e_1,e_2>$, then $\alpha_1(e_1)=e_4 \notin S$, which is a contradiction.\\
If $S=<e_1,e_3>$, then $\alpha_1(e_1)=e_4 \notin S$, which is a contradiction.\\
If $S=<e_1,e_4>$, then $\alpha_2(e_1)=e_1-e_2-e_3+e_4 \notin S$, which is a contradiction.\\
If $S=<e_2,e_3>$, then $\alpha_2(e_2)=-e_2+e_4 \notin S$, which is a contradiction.\\
If $S=<e_2,e_4>$, then $\alpha_1(e_2)=e_3 \notin S$, which is a contradiction.\\
If $S=<e_3,e_4>$, then $\alpha_1(e_3)=e_2 \notin S$, which is a contradiction.
\item[(c)] $S \neq <e_i,e_j,e_k>$ for all $1\leq i,j,k\leq 4$.\\
If $S=<e_1,e_2,e_3>$, then $\alpha_1(e_1)=e_4 \notin S$, which is a contradiction.\\
If $S=<e_1,e_2,e_4>$, then $\alpha_1(e_2)=e_3 \notin S$, which is a contradiction.\\
If $S=<e_1,e_3,e_4>$, then $\alpha_1(e_3)=e_2 \notin S$, which is a contradiction.\\
If $S=<e_2,e_3,e_4>$, then $\alpha_1(e_4)=e_1 \notin S$, which is a contradiction.
\end{itemize}
Thus $\hat{\phi}_B(t) \otimes \hat{\phi}_B(m)$ is irreducible in this case.

We then assume that $t \neq m$, $tm \neq 1$ and $t=1.$ It is then clear that $m\neq 1$. Suppose that $S$ is a non trivial invariant subspace of $\mathbb{C}^4$ under $\hat{\phi}_B(t) \otimes \hat{\phi}_B(m)$. In addition to the previous subspaces, we have other possible candidates to invariant subspaces.
\begin{itemize}
\item[(a)] If dim$(S)=1$, then we consider $S=<\beta_1e_1+\beta_3e_3>$ or $S=<\beta_2e_2+\beta_4e_4>$, where $\beta_i \neq 0$ for all $1\leq i \leq 4$.\\
We have $\epsilon_{32}(\beta_1e_1+\beta_3e_3)=\beta_1e_1+\beta_1(m^{-1}-1)e_2+\beta_3e_3+\beta_3(m^{-1}-1)e_4$ with $\beta_3(m^{-1}-1)\neq 0$, which is a contradiction.\\
Similarly, we have $\epsilon_{21}(\beta_2e_2+\beta_4e_4)=\beta_2(1-m^{-1})e_1+\beta_2m^{-1}e_2+\beta_4(1-m^{-1})e_3+\beta_4m^{-1}e_4$ with $\beta_2(1-m^{-1})\neq 0$, which is a contradiction.
\item[(b)] If dim$(S)=2$, then we consider $S=<\beta_1e_1+\beta_3e_3, \beta_2e_2+\beta_4e_4>$. Without loss of generality, we assume that all $\beta_i's$ are non zeros. We have $\epsilon_{32}(\beta_1e_1+\beta_3e_3)=\beta_1e_1+\beta_1(m^{-1}-1)e_2+\beta_3e_3+\beta_3(m^{-1}-1)e_4$, and so $\beta_1\beta_4=\beta_2\beta_3$. Now, $\alpha_2(\beta_1e_1+\beta_3e_3)=\beta_1e_1-\beta_1e_2+(-\beta_1-\beta_3)e_3+(\beta_1+\beta_3)e_4$, and so $-\beta_1-\beta_3=\beta_3$. Thus we get $\beta_3=-\frac{1}{2}\beta_1$. On the other hand, $\alpha_1(\beta_1e_1+\beta_3e_3)=\beta_3e_2+\beta_1e_4\in S$, so $\beta_3e_2+\beta_1e_4=k(\beta_2e_2+\beta_4e_4)$ for some non zero constant $k$. Thus we have $\beta_1\beta_2=\beta_3\beta_4$, which implies that $\beta_1\beta_2\beta_4=\beta_3\beta_4^2$. Having $\beta_1\beta_4=\beta_2\beta_3$, we get $\beta_2^2\beta_3=\beta_3\beta_4^2$ and so $\beta_2^2=\beta_4^2$. This means that $\beta_2= \pm \beta_4$, and so $\beta_1= \pm \beta_3$. This contradicts the fact that $\beta_3=-\frac{1}{2}\beta_1$.
\item[(c)] If dim$(S)=3$, then we assume, without loss of generality, that $S=<\beta_1e_1+\beta_3e_3, \beta_2e_2+\beta_4e_4, e_1>$, where $\beta_i \neq 0$ for all $1\leq i \leq 4$. Since we have $e_1 \in S$ and $\beta_1e_1+\beta_3e_3 \in S$, it follows that $e_3 \in S$. On the other hand, $\alpha_1(e_1)=e_4 \in S$ and $ \beta_2e_2+\beta_4e_4 \in S$, which implies that $e_2 \in S$ and so $S= \mathbb{C}^4$. This also gives a contradiction.
\end{itemize}
Thus $\hat{\phi}_B(t) \otimes \hat{\phi}_B(m)$ is irreducible in this case.

Now, we suppose that $t \neq m$ and $tm=1$. It follows that $t\neq 1$ and $m \neq 1$. Suppose $S$ is an invariant nontrivial subspace of $\mathbb{C}^4$ under $\hat{\phi}_B(t) \otimes \hat{\phi}_B(m)$.
\begin{itemize}
\item[(a)] $e_i \notin S$ for any $i=1,2,3,4.$\\
If $e_2 \in S$, then $\epsilon_{32}(e_2)= te_2+(1-t)e_4=X \in S$. Now, we have $X-te_2=(1-t)e_4 \in S$, which implies that $e_4 \in S$. Also, $\epsilon_{31}(e_2)=(t^{-1}-1)e_1+t^{-1}e_2=Y \in S$, and so $Y-t^{-1}e_2=(t^{-1}-1)e_1 \in S$. This implies that $e_1 \in S$.\\
So \begin{equation}
e_2 \in S \implies e_1,e_4 \in S 
\end{equation}
Similarly, if $e_3 \in S$, then $\epsilon_{21}(e_3)=(1-t^{-1})e_1+t^{-1}e_3=X \in S$. Now, we have $X-t^{-1}e_3=(1-t^{-1})e_1 \in S$, which implies that $e_1 \in S$. Also, $\epsilon_{32}(e_3)=t^{-1}e_3+(1-t^{-1})e_4= Y \in S$, and so $Y-t^{-1}e_3=(1-t^{-1})e_4 \in S$. This implies that $e_4 \in S$.\\
So \begin{equation}
e_3 \in S \implies e_1,e_4 \in S 
\end{equation}
Now, if $e_1 \in S$, then $\epsilon_{12}(e_1)=e_1+(t^{-1}-1)e_2+(t-1)e_3+(2-t-t^{-1})e_4=X \in S$, and $\epsilon_{32}(e_1)=e_1+(t-1)e_2+(t^{-1}-1)e_3+(2-t-t^{-1})e_4=Y \in S$. Then $X-Y=(-t+t^{-1})(e_2-e_3) \in S$, which implies that $e_2-e_3 \in S$. \\
So \begin{equation}
e_1 \in S \implies e_2-e_3 \in S 
\end{equation}
Similarly, if $e_4 \in S$, then $\epsilon_{21}(e_4)=(2-t-t^{-1})e_1+(t-1)e_2+(t^{-1}-1)e_3+e_4=X \in S$, and $\epsilon_{31}(e_4)=(2-t-t^{-1})e_1+(t^{-1}-1)e_2+(t-1)e_3+e_4=Y \in S$. Then $X-Y=(t-t^{-1})(e_2-e_3) \in S$, which implies that $e_2-e_3 \in S$. \\
So \begin{equation}
e_4 \in S \implies e_2-e_3 \in S. 
\end{equation}
Now, suppose that $e_2 \in S$. Then by (1), we have $e_1$ and $e_4 \in S$, and so by (3), we get $e_2-e_3 \in S$. This implies that $e_3 \in S$, and so $S=\mathbb{C}^4$, which is a contradiction. Hence $e_2 \notin S$. \\
Similarly, suppose that $e_3 \in S$. Then by (2), we have $e_1$ and $e_4 \in S$, and so by (3), we get $e_2-e_3 \in S$. This implies that $e_2 \in S$, and so $S=\mathbb{C}^4$, which is a contradiction. Hence $e_3 \notin S$. \\
Now, suppose that $e_1 \in S$. Then $\epsilon_{12}(e_1)-e_1=(t^{-1}-1)e_2+(t-1)e_3+(2-t-t^{-1})e_4=X \in S$ and $\epsilon_{32}(e_1)-e_1=(t-1)e_2+(t^{-1}-1)e_3+(2-t-t^{-1})e_4=Y \in S$. So $Z=X+Y=(2-t-t^{-1})(-e_2-e_3+2e_4) \in S$, which means that $W=-e_2-e_3+2e_4 \in S$. But $\epsilon_{31}(Y)+X=2(2-t-t^{-1})e_4 \in S$, which implies that $e_4 \in S$, and so $W-2e_4=-e_2-e_3 \in S$. By (3), we have $e_2-e_3 \in S$, so $e_2 \in S$, which is a contradiction. Hence $e_1 \notin S$.\\
Similarly, suppose that $e_4 \in S$. Then $\epsilon_{21}(e_4)-e_4=(2-t-t^{-1})e_1+(t-1)e_2+(t^{-1}-1)e_3=X \in S$ and $\epsilon_{31}(e_4)-e_4=(2-t-t^{-1})e_1+(t^{-1}-1)e_2+(t-1)e_3=Y \in S$. So $Z=X+Y=(2-t-t^{-1})(2e_1-e_2-e_3) \in S$, which means that $W=2e_1-e_2-e_3 \in S$. But $\epsilon_{23}(y)+X=2(2-t-t^{-1})e_1 \in S$, which implies that $e_1 \in S$, and then $W-2e_1=-e_2-e_3 \in S$. By (4), we have $e_2-e_3 \in S$, so $e_2 \in S$, which is a contradiction. Hence $e_4 \notin S$.\\
Therefore, $e_i \notin S$ for any $i=1,2,3,4.$
\item[(b)] $\alpha_ie_i+\alpha_je_j \notin S$ for any $1 \leq i\neq j \leq 4$.\\
Suppose that $\alpha_1e_1+\alpha_2e_2 \in S$ with $\alpha_1 \neq 0$ and $\alpha_2 \neq 0$. Then $\epsilon_{23}(\alpha_1e_1+\alpha_2e_2)=\alpha_1e_1+t\alpha_2e_2=X \in S$, and so $X-(\alpha_1e_1+\alpha_2e_2) \in S$. This implies that $-(t-1)\alpha_2e_2 \in S$. And so $e_2 \in S$, which contradicts (a). Thus  $\alpha_1e_1+\alpha_2e_2 \notin S$.\\
\\
Suppose that $\alpha_1e_1+\alpha_3e_3 \in S$ with $\alpha_1 \neq 0$ and $\alpha_3 \neq 0$. Then $\epsilon_{23}(\alpha_1e_1+\alpha_3e_3)=\alpha_1e_1+t^{-1}\alpha_3e_3=X \in S$, and so $X-(\alpha_1e_1+\alpha_3e_3) \in S$. This implies that $-(t^{-1}-1)\alpha_3e_3 \in S$. Then we get $e_3 \in S$, this contradicts (a). Thus  $\alpha_1e_1+\alpha_3e_3 \notin S$.\\
\\
Suppose that $\alpha_2e_2+\alpha_3e_3 \in S$, with $\alpha_2 \neq 0$ and $\alpha_3 \neq 0$. Then $\epsilon_{32}(\alpha_2e_2+\alpha_3e_3)=t\alpha_2e_2+t^{-1}\alpha_3e_3+((1-t)\alpha_2+(1-t^{-1})\alpha_3)e_4=X_1 \in S$ and $\epsilon_{23}(\alpha_2e_2+\alpha_3e_3)=t\alpha_2e_2+t^{-1}\alpha_3e_3=Y_1 \in S$. So $Z_1=X_1-Y_1=((1-t)\alpha_2+(1-t^{-1})\alpha_3)e_4 \in S$. On the other hand, $\epsilon_{12}(\alpha_2e_2+\alpha_3e_3)=t^{-1}\alpha_2e_2+t\alpha_3e_3+((1-t^{-1})\alpha_2+(1-t)\alpha_3)e_4=X_2 \in S$ and  $\epsilon_{13}(\alpha_2e_2+\alpha_3e_3)=t^{-1}\alpha_2e_2+t\alpha_3e_3=Y_2 \in S$. And so $Z_2=X_2-Y_2=((1-t^{-1})\alpha_2+(1-t)\alpha_3)e_4 \in S$. Hence, $Z_1+Z_2=(-t-t^{-1}+2)(\alpha_2+\alpha_3)e_4 \in S$. But $e_4 \notin S$ by (a), then $\alpha_2+\alpha_3=0$, and so $\epsilon_{13}(\alpha_2e_2+\alpha_3e_3)-\epsilon_{31}(\alpha_2e_2+\alpha_3e_3)=(t-t^{-1})e_1 \in S$. This implies that $e_1 \in S$, which contradicts (a). Thus $\alpha_2e_2+\alpha_3e_3 \notin S$.\\
\\
Suppose that $\alpha_1e_1+\alpha_4e_4 \in S$ with $\alpha_1 \neq 0$ and $\alpha_4 \neq 0$. Then $\epsilon_{12}(\alpha_1e_1+\alpha_4e_4)=\alpha_1e_1+\alpha_1(t^{-1}-1))e_2+\alpha_1(t-1)e_3+(\alpha_1(2-t-t^{-1})+\alpha_4)e_4=X \in S$ and $\epsilon_{32}(\alpha_1e_1+\alpha_4e_4)=\alpha_1e_1+\alpha_1(t-1)e_2+\alpha_1(t^{-1}-1)e_3+(\alpha_1(2-t-t^{-1})+\alpha_4)e_4=Y \in S$. So $X-Y= \alpha_1(-t+t^{-1})e_2+\alpha_1(-t^{-1}+t)e_3 \in S$, which contradicts the previous result. Thus $\alpha_1e_1+\alpha_4e_4 \notin S$.\\
\\
Suppose that $\alpha_2e_2+\alpha_4e_4 \in S$ with $\alpha_2 \neq 0$ and $\alpha_4 \neq 0$. Then $\epsilon_{23}(\alpha_2e_2+\alpha_4e_4)-(\alpha_2e_2+\alpha_4e_4)=(t-1)\alpha_2e_2\in S$, and so $e_2 \in S$, which contradicts (a). Thus $\alpha_2e_2+\alpha_4e_4 \notin S$.\\
\\
Suppose that $\alpha_3e_3+\alpha_4e_4 \in S$ with $\alpha_3 \neq 0$ and $\alpha_4 \neq 0$. Then $\epsilon_{23}(\alpha_3e_3+\alpha_4e_4)-(\alpha_3e_3+\alpha_4e_4)=(t^{-1}-1)\alpha_3e_3\in S$, and so $e_3 \in S$, which contradicts (a). So $\alpha_3e_3+\alpha_4e_4 \notin S$.\\
Therefore, $\alpha_ie_i+\alpha_je_j \notin S$ for any $1 \leq i\neq j \leq 4$.
\item[(c)] $\alpha_ie_i+\alpha_je_j+\alpha_ke_k \notin S$ for any $1 \leq i\neq j\neq k \leq 4$.\\
Suppose that $\alpha_1e_1+\alpha_2e_2+\alpha_3e_3 \in S$ with $\alpha_1 \neq 0, \alpha_2 \neq 0,$ and $\alpha_3 \neq 0$. Then $\epsilon_{13}(\alpha_1e_1+\alpha_2e_2+\alpha_3e_3)-(\alpha_1e_1+\alpha_2e_2+\alpha_3e_3)= (t-1)\alpha_2e_2+(t^{-1}-1)\alpha_3e_3 \in S$, which contradicts (b). So $\alpha_1e_1+\alpha_2e_2+\alpha_3e_3 \notin S$.\\
\\
Suppose that $\alpha_1e_1+\alpha_2e_2+\alpha_4e_4 \in S$ with $\alpha_1 \neq 0, \alpha_2 \neq 0,$ and $\alpha_4 \neq 0$. Then $\epsilon_{13}(\alpha_1e_1+\alpha_2e_2+\alpha_4e_4)-(\alpha_1e_1+\alpha_2e_2+\alpha_4e_4)= (t-1)\alpha_2e_2 \in S$, and so $e_2 \in S$, which contradicts (a). So $\alpha_1e_1+\alpha_2e_2+\alpha_4e_4 \notin S$.\\
\\
Suppose that $\alpha_1e_1+\alpha_3e_3+\alpha_4e_4 \in S$ with $\alpha_1 \neq 0, \alpha_3 \neq 0,$ and $\alpha_4 \neq 0$. Then $\epsilon_{13}(\alpha_1e_1+\alpha_3e_3+\alpha_4e_4)-(\alpha_1e_1+\alpha_3e_3+\alpha_4e_4)= (t^{-1}-1)\alpha_3e_3 \in S$, and so $e_3 \in S$, which contradicts (a). So $\alpha_1e_1+\alpha_3e_3+\alpha_4e_4 \notin S$.\\
\\
Suppose that $\alpha_2e_2+\alpha_3e_3+\alpha_4e_4 \in S$ with $\alpha_2 \neq 0, \alpha_3 \neq 0,$ and $\alpha_4 \neq 0$. Then $\epsilon_{13}(\alpha_2e_2+\alpha_3e_3+\alpha_4e_4)-(\alpha_2e_2+\alpha_3e_3+\alpha_4e_4)= (t-1)\alpha_2e_2+(t^{-1}-1)\alpha_3e_3 \in S$, which contradicts (b). So $\alpha_2e_2+\alpha_3e_3+\alpha_4e_4 \notin S$.\\
Therefore, $\alpha_ie_i+\alpha_je_j+\alpha_ke_k \notin S$ for any $1 \leq i\neq j\neq k \leq 4$.
\item[(d)] $\alpha_1e_1+\alpha_2e_2+\alpha_3e_3+\alpha_4e_4 \notin S$.\\
Suppose that $\alpha_1e_1+\alpha_2e_2+\alpha_3e_3+\alpha_4e_4 \in S$. Then $\epsilon_{13}(\alpha_1e_1+\alpha_2e_2+\alpha_3e_3+\alpha_4e_4)-(\alpha_1e_1+\alpha_2e_2+\alpha_3e_3+\alpha_4e_4)= (t-1)\alpha_2e_2+(t^{-1}-1)\alpha_3e_3 \in S$, which contradicts (b).\\
Therefore, $\alpha_1e_1+\alpha_2e_2+\alpha_3e_3+\alpha_4e_4 \notin S$.
\end{itemize}
Thus, $\mathbb{C}^4$ contains no nontrivial invariant subspace under $\hat{\phi}_{B}(t) \otimes \hat{\phi}_{B}(m)$. Therefore, $\hat{\phi}_{B}(t) \otimes \hat{\phi}_{B}(m)$ is irreducible.
\end{proof}


\end{document}